\documentclass[a4paper]{article}%
\usepackage{amsmath}
\usepackage{amsfonts}%
\usepackage{amssymb}
\newtheorem{theorem}{Theorem}[section]

\newtheorem{lemma}[theorem]{Lemma}

\newtheorem{remark}[theorem]{Remark}

 \newcommand{\C}{\mathbb C}

 \newcommand{\ex}{ {\rm e} }
 
\newenvironment{proof}[1][Proof]{\noindent \emph{#1.} }{\hfill \ \rule{0.5em}{0.5em}}

\makeatletter\@addtoreset{equation}{section}\makeatother
\makeatletter\@addtoreset{figure}{section}\makeatother
\makeatletter\@addtoreset{table}{section}\makeatother
\begin{document}
\title{Two-points problem for an evolutional first order equation in Banach space}
\author{T.Ju.Bohonova \thanks{National Aviation University of Ukraine,
 1, Komarov ave. , 03058 Kyiv, Ukraine
({\tt bohonoff@astral.kiev.ua}).},  V.B.Vasylyk
\thanks{Institute of Mathematics of NAS of Ukraine,
 3 Tereshchenkivs'ka Str., Kyiv-4, 01601, Ukraine
({\tt vasylyk@imath.kiev.ua}).}
 }

\date{\null}
\maketitle

\begin{abstract}
 Two-point nonlocal  problem for the first order differential evolution equation with an operator coefficient in a Banach space $X$ is considered. An exponentially convergent algorithm is proposed and justified in assumption that the operator coefficient is strongly positive and some existence and uniqueness conditions are fulfilled. This algorithm leads to a system of linear equations that can be solved by fixed-point iteration. The algorithm provides exponentially convergence in time that in combination with fast algorithms on spatial variables can be efficient treating such problems. The efficiency of the proposed algorithms is demonstrated by numerical examples.
\end{abstract}



\section{Introduction}

The m-point initial (nonlocal) problem for a differential equation with the nonlocal  condition $u(t_0) + g(t_1;\dots ; t_p; u) = u_0$ and a given function $g$ on a given point set  $P=\{0= t_0 < t_1 <\dots < t_p\}$ , is one of the important topics in the study of differential equations. Interest in such problems originates mainly from some physical problems with a control of the solution at $P$. For example, when the function $g(t_1;\dots ; t_p; u)$ is linear, we will have the periodic problem $u(t_0)=u(t_1)$. Problems with nonlocal conditions arise in the theory of physics of plasma \cite{samde}, nuclear physics \cite{LeunCh}, mathematical chemistry \cite{LeuOrt}, waveguides \cite{Gord} etc. Two-point problem is also useful for considering the finale value problem \cite{ShowFVP}.

Differential equations with operator coefficients in some Hilbert or Banach space can be considered as meta-models for systems of partial or ordinary differential equations and are suitable for investigations using tools of the functional analysis (see e.g. \cite{clem, krein}). Nonlocal problems can also be considered within this framework \cite{Byszewski2, Byszewski3}.

Discretization methods for differential equations in Banach and Hilbert spaces were intensively studied in the last decade (see e.g. \cite{fernandez-lubich-palencia-schaedle, ghk, ghk3, gm2, lopez-fernandez2, lopez-fernandez1, SSloanTho:98, sst1, vas1, vas2} and the references therein). Methods from \cite{ ghk, ghk3,  gm2, lopez-fernandez2, lopez-fernandez1, sst1, vas1, vas2} possess an exponential convergence rate, i.e. the error estimate in an appropriate norm is of the type ${\cal O}(\ex^{-N^\alpha}),$ $\alpha>0$ with respect to a discretization parameter $N \to \infty$. For a given tolerance $\varepsilon$ such discretization provides optimal or nearly optimal computational complexity \cite{ghk}.

In the present paper we consider the problem

\begin{equation}\label{nle-1}
\begin{split}
& \frac{d u(t)}{d t}+A_1(t)u(t)=f_1(t), \\
& u(0)+\alpha u(1)=\varphi ,
\end{split}
\end{equation}
where $A_1(t)$ is a densely defined closed (unbounded) operator with the domain $D(A_1)$ independent of $t$ in a Banach space $X$, $\varphi$ is given vector and $f_1(t)$ is given vector-valued function, $\alpha \in \mathbb{R}.$ We suppose that the operator $A_1(t)$ is strongly positive; i.e. there exists a positive constant $M_R$ independent of $t$ such that on the rays and outside a sector $\Sigma_\theta =\{z \in \C: 0 \le arg(z)\le \theta, \theta \in (0,\pi/2)\}$ the following resolvent estimate holds:
\begin{equation} \label{as1}
\|(zI-A_1(t))^{-1}\|\le \frac{M_R}{1+|z|}.
\end{equation}
This assumption implies that there exists a positive constant
$c_\kappa$ such that ( see \cite{fujita}, p.103)
\begin{equation} \label{as2}
\|A_1^\kappa(t)e^{-sA_1(t)}\|\le c_\kappa s^{-\kappa}, \quad s>0,
\quad \kappa \ge 0 .
\end{equation}
Our further assumption is that there exists a real positive
$\omega$ such that
\begin{equation} \label{as3}
\|e^{-sA_1(t)}\|\le e^{- \omega s} \quad \forall s,\; t \in [0,1]
\end{equation}
(see \cite{pazy}, Corollary 3.8, p.12, for corresponding
assumptions on $A_1(t)$). Let us also assume that the following conditions hold true

\begin{equation} \label{as4}
\|[A_1(t)-A_1(s)]A_1^{-\gamma}(t)\| \le L_{1,\gamma}|t-s| \quad
\forall t,\; s, \; 0\le \gamma < 1,
\end{equation}

\begin{equation} \label{as5}
\|A_1^\gamma(t)A_1^{-\gamma}(s)-I\|\le L_\gamma |t-s| \quad
\forall t,\; s \in [0,1].
\end{equation}

We suppose also that
\begin{equation}\label{cond-f}
	f_1(t) \in C(0,1;X).
\end{equation}

The aim of this paper is to construct an exponentially convergent approximation to the problem \eqref{nle-1} for a differential equation with two-points nonlocal condition in abstract setting. The paper is organized as follows. In Section \ref{sect2} we discuss the existence and uniqueness of the solution as well as its representation through input data.  A numerical algorithm is presented in section \ref{numalg}. The main result of this section is theorem \ref{nle-main0} about the convergence rate of the proposed discretization. The next section \ref{sect6} we represent some numerical example which confirm theoretical results from the previous sections.

\section{Existence and uniqueness of the solution}\label{sect2}

It is well known, that for $\alpha=0$ the problem \eqref{nle-1} has a unique solution under the assumptions \eqref{as1}-\eqref{cond-f} (se e.g. \cite{pazy, krein}). This solution can be write down as follows:
\begin{equation}\label{locrep}
 u(t)=U(t,0)u(0)+\int_0^t U(t,s) f_1(s)ds=U(t,0)\varphi +\int_0^t U(t,s) f_1(s)ds,
\end{equation}
where $U(t,s)$ is an evolution operator that corresponds to \eqref{nle-1} for $\alpha=0.$

 Let us study conditions when there is unique solution for the two-points problem \eqref{nle-1}. We have from \eqref{locrep}
\[
u(1)=U(1,0)u(0)+\int_0^1 U(1,s) f_1(s)ds.
\]
Substituting this expression into the nonlocal condition we obtain
\[
u(0)=\left[ I+ \alpha U(1,0) \right]^{-1} \left[ \varphi -\alpha \int_0^1 U(1,s) f_1(s) ds \right],
\]
and for $u(t)$ we have
\begin{equation}\label{solrep}
	u(t)=U(t,0)\left[ I+ \alpha U(1,0) \right]^{-1} \left[ \varphi -\alpha \int_0^1 U(1,s) f_1(s) ds \right] +\int_0^t U(t,s) f_1(s)ds.
\end{equation}

It is necessary to establish conditions on $\alpha$ for the existence of $u(t).$ In fact, we have to explore when exists $\left[ I+ \alpha U(1,0) \right]^{-1}.$ So, we obtain using estimate for $U(t,s)$ (see e.g. \cite{pazy, krein}).
\[
\left\| \left[ I+ \alpha U(1,0) \right]^{-1} \right\| \le \left[ 1- |\alpha| \left\|U(1,0) \right\| \right] ^{-1} \le \left[ 1- |\alpha| M \right] ^{-1} \le C,
\]
for small enough $\alpha $ ($\alpha < M^{-1}$).

\section{Numerical algorithm}\label{numalg}

We use the approach developed in \cite{gm1} and \cite{bgmv} to construct numerical method for solving problem \eqref{nle-1}. First of all we change variable in \eqref{nle-1} by $ t \to \frac{1+t}{2}$ and for $v(t)=u\left(\frac{1+t}{2}\right)$ we have
\begin{equation}\label{nle-post2}
	\begin{split}
& \frac{d v(t)}{d t}+A(t)v(t)=f(t), \\
& v(-1)+\alpha v(1)=\varphi, 
 \end{split}
\end{equation}
where $A(t)=\frac{1}{2}A_1\left(\frac{1+t}{2}\right),$ $f(t)=\frac{1}{2} f_1\left(\frac{1+t}{2}\right),$  
 
 We choose a mesh $\omega_{n}=\{t_k, \; k=0,...,n\}$ of $n+1$ various
points  on $[-1,1]$ that are Chebyshev-Gauss-Lobatto nodes $t_k= \cos\left(\frac{n-k}{n}\pi\right),$ and set $\tau_k=t_k-t_{k-1}.$ Let
\begin{equation} \label{nle-fs1}
\begin{split}
\overline{A}(t)=&A_k=A(t_k), t \in (t_{k-1},t_k], \quad k=\overline{1,n},\\
&A_0=A(-1).
\end{split}
\end{equation}
Let us rewrite the problem (\ref{nle-post2}) in the equivalent form
\begin{equation} \label{nle-fs2}
\begin{split}
& \frac{dv}{dt}+\overline{A}(t)v=[\overline{A}(t)-A(t)]v(t)+f(t), \quad t\in (-1,1) \\
&v(-1)=\varphi -\alpha v(1).
\end{split}
\end{equation}

Note, that now all operators on the left side of these equations are constant on each subinterval and piece-wise constant on the whole interval $[-1,1].$

On each subinterval we can write down the equivalent to \eqref{nle-fs2} integral equation
\begin{equation}\label{nle-pwsol-n}
\begin{split}
	v(t)=&\ex^{-A_k(t-t_{k-1})}v(t_{k-1}) +\int_{t_{k-1}}^t \ex^{-A_k(t-s)}\left[ A_k -A(t)\right]v(s) ds +\int_{t_{k-1}}^t \ex^{-A_k(t-s)}f(s) ds, \quad \\ 
	&t\in [t_{k-1},t_k], \quad	k=\overline{2,n}, 
\end{split}	
\end{equation}
\begin{equation}\label{nle-pwsol-1}
	v(t)=\ex^{-A_1(t+1)}\left[\varphi -\alpha v(1) \right] +\int_{-1}^t \ex^{-A_1(t-s)}\left[ A_1 -A(t)\right]v(s) ds +\int_{-1}^t \ex^{-A_1(t-s)}f(s) ds, \quad t\in [-1,t_1]. 
\end{equation}

 Let
\begin{equation}\label{nle-s4}
P_{n}(t;v)=P_{n}v=\sum_{j=0}^n v(t_j) L_{j,n}(t), \\
\end{equation}
be the interpolation polynomial for $v(t)$ on the mesh $\omega_n$, $x=( x_0,...,x_n), x_i \in X$ given vector and
\begin{equation}\label{nle-s5}
 P_{n}(t;y)=P_{n}x=\sum_{j=0}^n x_j L_{j,n}(t)
\end{equation}
the polynomial that interpolates $x$, where
\[
L_{j,n}(s)=\frac{T_{n}^{\prime }(s)(1-s^{2})}{\frac{d}{ds}[(1-s^{2})T_{n}^{\prime }(s)]_{s=s_{j}}(s-s_{j})},\quad j=0,...,n
\]
 are the  Lagrange fundamental polynomials. Substituting $P_{n}(s;x)$ for $v(s)$, $x_k$ for $v(t_k)$ and then setting $t=t_k$ in (\ref{nle-pwsol-n}) we arrive at the following system of linear equations with respect to the unknown $x_k:$
\begin{equation} \label{nle-slar}
\begin{split}
&x_0 +\alpha x_n=\varphi,\\
&x_k=\ex^{-A_k \tau_k}x_{k-1} +\sum_{j=0}^n \alpha_{kj}x_j +\phi_k, \quad k=\overline{1,n},
\end{split}
\end{equation}
which represents our algorithm. Here we use the notations
\begin{equation} \label{nle-poz}
\begin{split}
&\alpha_{kj}=\int_{t_{k-1}}^{t_k}\ex^{-A_k (t_k-s)}[A_k-A(s)]L_{j,n}(s)ds,  \\
&\phi_k =\int_{t_{k-1}}^{t_k}\ex^{-A_{k}(t_{k} -s)}f(s)ds, \quad k=\overline{1,n}, \quad j=\overline{0,n},
\end{split}
\end{equation}
and suppose that we have an algorithm to compute these
coefficients.
 
 For the error $z=(z_{1},...,z_{n})$, with $z_{k}=v(t_k)-x_k$ we have the relations
\begin{equation} \label{nle-poh}
\begin{split}
&z_0+\alpha z_n =0, \\
&z_{k}=\ex^{-A_k \tau_k}z_{k-1} +\sum_{j=0}^n \alpha_{kj}z_{j} +\psi_{k}, \quad k=\overline{1,n},
\end{split}
\end{equation}
where
\begin{equation}\label{nle-poz2}
\psi_{k}= \int_{t_{k-1}}^{t_k}\ex^{-A_k(t_k -s)}[A_k -A(s)][v(s) -P_n(s;v)]ds, \quad k=\overline{1,n},
\end{equation}

In order to represent algorithm (\ref{nle-slar}) in a block-matrix form
we introduce the matrix
\begin{equation}\label{nle-ms}
S=
  \begin{pmatrix}
    I & 0 & 0 & \cdot & \cdot & \cdot & 0 & \alpha \sigma_0 \\
    -\sigma_1 & I & 0 & \cdot & \cdot & \cdot & 0 & 0 \\
    0 & -\sigma_2 & I & \cdot & \cdot & \cdot & 0 & 0 \\
    \cdot & \cdot & \cdot & \cdot & \cdot & \cdot & \cdot & \cdot \\
    0 & 0 & 0 & \cdot & \cdot & \cdot & -\sigma_{n} & I \
  \end{pmatrix},
\end{equation}
where $\sigma_0=A_0^\gamma A_n^{-\gamma},$ $\sigma_k=\ex^{-A_k \tau_k}A_k^\gamma A_{k-1}^{-\gamma},$ $k=\overline{1,n},$ the matrix $B=\{\tilde{\alpha}_{k,j}\}_{k,j=0}^n$ with $\tilde{\alpha}_{k,j}=A_k^\gamma \alpha_{k,j}A_j^{-\gamma},$ $k=\overline{1,n},$ $j=\overline{0,n},$ and $\tilde{\alpha}_{0,j}=0,$ $j=\overline{0,n},$  the vectors
\begin{equation}\label{nle-vect}
  \tilde{x}=\begin{pmatrix}
    A_0^\gamma x_0 \\
    A_1^\gamma x_1 \\
    \cdot \\
    \cdot \\
    A_n^\gamma x_n \
  \end{pmatrix},\;
  \phi=\begin{pmatrix}
    A_0^\gamma \varphi \\
    A_1^\gamma \phi_1 \\
    \cdot \\
   \cdot \\
    A_n^\gamma \phi_n \
  \end{pmatrix},\;
  \tilde{z}=\begin{pmatrix}
    A_0^\gamma z_0 \\
    A_1^\gamma z_1 \\
    \cdot \\
    \cdot \\
    A_n^\gamma z_n \
  \end{pmatrix},\;
  \psi=\begin{pmatrix}
    0 \\
    A_1^\gamma \psi_1 \\
    \cdot \\
    \cdot \\
    A_n^\gamma \psi_n \
  \end{pmatrix}.
\end{equation}

It is easy to check that for the (left) inverse
\begin{equation}\label{nle-Sinv}
	S^{-1}=\delta \left(R_1-R_2\right),
\end{equation}
where
\[
\delta= \left(I+ \alpha \sigma_0\sigma_1 \dots \sigma_n \right)^{-1},
\]
\[
R_1=
\begin{pmatrix}
 I & 0&   \cdots & 0 & 0 \\
 {\sigma}_{1} &I  &  \cdots & 0 & 0 \\
 {\sigma}_{2} {\sigma}_{1} &{\sigma}_{2}  &  \cdots & 0 & 0 \\
  \cdot & \cdot &  \cdots & \cdot & \cdot \\
 {\sigma}_{n} \cdots {\sigma}_{1} & {\sigma}_{n}\cdots {\sigma}_{2}
   & \cdots & {\sigma}_{n} & I
\end{pmatrix},
\]
\[
R_2= \alpha s_0
\begin{pmatrix}
 0 & \sigma_n \dots \sigma_2& \sigma_n \dots \sigma_3 &    \cdots & \sigma_n & I \\
 0 & 0 &\sigma_1\sigma_n\dots \sigma_3  &  \cdots & \sigma_1 \sigma_n & \sigma_1 \\
 \cdot & \cdot & \cdot & \cdots & \cdot & \cdot \\
 0 & 0 & 0  &  \cdots & 0 & \sigma_{n-1}\dots \sigma_1 \\
 0& 0 & 0  & \cdots & 0 & 0
\end{pmatrix}.
\]

\begin{remark}
Using results of \cite{ gm5, ghk,ghk3} one can get a parallel and sparse approximations with an exponential convergence rate of the operator exponentials contained in ${S}^{-1}$ and as a consequence a parallel and sparse approximation of ${S}^{-1}$.
\end{remark}

We multiply the equations in  \eqref{nle-slar} and the equation in \eqref{nle-poh} by $A_k^\gamma ,$ $k=\overline{0,n}$ and obtain
\begin{equation} \label{nle-nslar}
\begin{split}
& A_0^\gamma x_0 +\alpha A_0^\gamma x_n =A_0^\gamma \varphi , \\
& A_k^\gamma x_k= \ex^{-A_k \tau_k} A_k^\gamma x_{k-1} + \sum_{j=0}^n \tilde{\alpha}_{kj}A_j^\gamma x_j + A_k^\gamma \phi_k,  \quad k=\overline{1,n},
\end{split}
\end{equation}

\begin{equation} \label{nle-npoh}
\begin{split}
& A_0^\gamma z_0 +\alpha A_0^\gamma z_n =0, \\
& A_k^\gamma z_k= \ex^{-A_k \tau_k} A_k^\gamma z_{k-1} + \sum_{j=0}^n \tilde{\alpha}_{kj}A_j^\gamma z_j + A_k^\gamma \psi_k,  \quad k=\overline{1,n},
\end{split}
\end{equation}
Then the systems \eqref{nle-nslar}, \eqref{nle-npoh} can be written down in the matrix form using notations \eqref{nle-ms}, \eqref{nle-vect}  as
\begin{equation} \label{nle-mform}
\begin{split}
&S \tilde{x} = B \tilde{x}+ \phi,\\
&S \tilde{z} = B \tilde{z}+ \psi.
\end{split}
\end{equation}

Next, for a vector $v=(v_1,v_2,...,v_n)^T$ and a block operator matrix $A=\{a_{ij}\}_{i,j=1}^n$ we introduce a vector norm
\[
|\|v\|| \equiv |\|v\||_1=\max_{1 \le k \le n} \|v_k\|,
\]
and the consistent matrix norm
\[
|\|A\|| \equiv |\|A\||_1 =\max_{1 \le i \le n}\sum_{j=1}^n
\|a_{i,j}\|.
\]

Due to \eqref{as5} we have $|\| A_{k}^{\gamma}A_{k-1}^{-\gamma}\| |= | \|A_{k}^{\gamma}A_{k-1}^{-\gamma}-I+I|\|\le 1+L_\gamma \tau_k,$ $\| \sigma_0 \| =\|A_0^\gamma A_n^{-\gamma}\|\le 1+L_\gamma T.$ In our case $T=2.$ So, we have the following, using these estimates
\[
\|\sigma_k\| = \| \ex^{-A_k \tau_k} A_{k}^{\gamma}A_{k-1}^{-\gamma}\| \le \ex^{-\omega \tau_k} \|A_{k}^{\gamma}A_{k-1}^{-\gamma} \|\le \ex^{-\omega \tau_k} \left(1+L_\gamma \tau_k \right),
\]
\[
\| \delta \| =\| \left(I+ \alpha \sigma_0\sigma_1 \dots \sigma_n \right)^{-1} \|\le \left( 1-| \alpha | \, \|\sigma_0\| \, \|\sigma_1\| \, \|\sigma_2\| \dots \|\sigma_n\|  \right)^{-1} 
\]
\[
 \le \left( 1-| \alpha | \, \left( 1+2L_\gamma \right) \, \ex^{-\omega \tau_1} \left(1+L_\gamma \tau_1 \right) \, \ex^{-\omega \tau_2} \left(1+L_\gamma \tau_2 \right) \dots \ex^{-\omega \tau_n} \left(1+L_\gamma \tau_n \right)  \right)^{-1}
\]
\[
 \le \left( 1-| \alpha | \, \left( 1+2L_\gamma \right) \, \ex^{-2\omega } \left(1+\frac{2L_\gamma}{n} \right)^n  \right)^{-1}\le \left( 1-| \alpha | \, \left( 1+2L_\gamma \right) \, \ex^{-2\omega } \ex^{2L_\gamma }  \right)^{-1} \le c,
\]
for $\alpha$ small enough.

In order to estimate the norm of matrix $S$ we have to estimate the norms of matrices $R_1,$ $R_2.$ In \cite{gm1} it was proved that for matrix similar to $R_1$ the estimate $|\| R_1\| |\le cn$ holds true. Let us estimate the norm of matrix $R_2.$
\[
\begin{split}
& |\| R_2 \|| \le \left(1+2c\right)\left(1+e^{-\omega \tau}(1+c
\tau)+\cdots+[e^{-\omega \tau}(1+c \tau)]^{n-1}\right)\\
&\le \left(1+2c\right) \left(1+(1+c \tau)+\cdots+(1+c \tau)^{n-1}\le \frac{(1+c
\tau)^{n}-1}{c \tau}\right)\le \left(1+2c\right)\frac{e^{2c}}{c \tau}\le cn.
\end{split}
\]

Using these estimates we obtain that
\begin{equation}\label{nle-ocS}
	|\|S^{-1}\||\le c n.
\end{equation}
 
 It was proved an estimate for the matrix $B$ in \cite{gm1}:
 \begin{equation}\label{nle-ocB}
	|\|B\||\le c n^{\gamma -2}\ln(n).
\end{equation}
So we can formulate the following assertion

\begin{lemma}\label{nle-l-oc}
 Let the assumptions \eqref{as1}-\eqref{as5} are fulfilled then the estimates \eqref{nle-ocS}, \eqref{nle-ocB} hold true.
\end{lemma}

Using \eqref{nle-mform} we have
\begin{equation} \label{nle-msol}
\begin{split}
& \tilde{x} = \left[ E- S^{-1}B \right]^{-1} S^{-1} \phi,\\
& \tilde{z} = \left[ E- S^{-1}B \right]^{-1} S^{-1} \psi,
\end{split}
\end{equation}
where $E$ is a diagonal matrix with unit operators $I$ on diagonal. Using lemma \ref{nle-l-oc} we obtain that 
\begin{equation}\label{nle-ocsb}
	|\|S^{-1}B\||\le c n^{\gamma -1}\ln(n)\to 0, \; n\to \infty.
\end{equation}
It means that for $n$ large enough there exists the matrix $\left[ E- S^{-1}B \right]^{-1}$ and 
\[
 \left|\left\| \left[ E- S^{-1}B \right]^{-1} \right\|\right|\le c.
\]
Consequently we obtain the following stability estimates from \eqref{nle-msol} using lemma \ref{nle-l-oc}:
\begin{equation}\label{nle-stab}
\begin{split}
	&|\| \tilde{x} \|| \le cn|\| \phi \|| ,\\
	&|\| \tilde{z} \|| \le cn|\| \psi \|| .
\end{split}	
\end{equation}
 
Let $\Pi_{n}$ be the set of all polynomials in $t$ with vector coefficients of degree less or equal then $n.$ In complete analogy with \cite{babenko, szegoe, szegoe1} the following Lebesgue inequality for vector-valued functions can be proved
\begin{equation}\label{nle-ocInt}
\| u(t)-P_{n}(t; u)\|_{C[-1,1]}\equiv \max_{t \in [-1,1]}\| u(t)-P_{n}(t; u)\| \le (1+\Lambda_n)E_n(u),
\end{equation}
with the error of the best approximation of $u$ by polynomials of degree not greater then $n$
\begin{equation}\label{nle-bestapp}
E_n( u)=\underset{p \in \Pi_{n}}{\text{inf}}\max_{t \in [-1,1]}\| u(t)-p(t)\|.
\end{equation}

Now, we can go over to  the main result of this section.

\begin{theorem}\label{nle-main0}
Let the assumptions of Lemma \ref{nle-l-oc} with $\gamma<1$ hold, then there exists a positive constant $c$ such that
\begin{enumerate}
\item For $n$ large enough it holds
\begin{equation}\label{nle-ocpoh}
|\| \tilde{z} \||\le c n^{\gamma-1}\cdot \ln{n} \cdot E_n(A_0^\gamma v),\\
\end{equation}
where $v$ is the solution of \eqref{nle-post2};
  \item The first equation in  \eqref{nle-mform} can be written in the form
\begin{equation}\label{nle-pr-fpiter}
 \tilde{x}=S^{-1}B\tilde{x}+S^{-1} \phi ,
\end{equation}
and can be solved by the fixed point iteration
\begin{equation}\label{nle-fpiter}
 \tilde{x}^{(k+1)}=S^{-1}B\tilde{x}^{(k)}+ S^{-1} \phi ,\; k=0,1,...; \; \tilde{x}^{(0)}- arbitrary,
\end{equation}
with the convergence rate of an geometrical progression with the denominator $q \le c n^{\gamma -1}\ln(n)<1$ for $n$ large enough.
\end{enumerate}
\end{theorem}

\begin{proof}
For $\tilde{z}$ we have the second estimate in \eqref{nle-stab}. The norm of the first summand on the right hand side of this inequality can be estimated in the following way

\[
|\|\psi \||=\max_{1 \le k \le n} \left\| \int_{t_{k-1}}^{t_k} \left\{ A_k^\gamma \ex^{-A_k (t_k -s)} [A_k-A(s)]A_k^{-\gamma}(A_k^{\gamma} A_0^{-\gamma})(A_0^\gamma v(s) -P_n(s;A_0^\gamma v))\right\} d s \right\| 
\]
\[
\le c \max_{1 \le k \le n}  \int_{t_{k-1}}^{t_k} |t_k -s|^{-\gamma} |t_k -s|\, \|A_0^\gamma v(s)-P_n(s;A_0^\gamma v)\|  d s  
\]
\[
\le c \tau_{max}^{2-\gamma}\left\|A_0^\gamma u(s)-P_{n}(\cdot;A_0^\gamma v)\right\|_{C[-1,1]} \le c \tau_{max}^{2-\gamma}(1+\Lambda_n)E_n(A_0^\gamma v).
\]

So, we obtain
\begin{equation}\label{nle-ocpsi}
	|\|\psi \|| \le c n^{\gamma-2}\cdot \ln{n} \cdot E_n(A_0^\gamma u),
\end{equation}

Now, the first assertion of the theorem follows from \eqref{nle-stab},
\eqref{nle-ocpsi}. The second one follows from \eqref{nle-mform} and \eqref{nle-ocsb}. \end{proof}

\section{Examples} \label{sect6}

Let us consider the following problem
\begin{equation}\label{nle-pr}
\begin{split}
&\frac{\partial u(x,t)}{\partial t} -\frac{\partial^2 u(x,t)}{\partial x^2} +q(x,t)u(x,t)=f(x,t), \\ 
&u(0,t)=u(1,t)=0,\\
&u(x,-1)+\alpha u(x,1)=\varphi(x),
\end{split}
\end{equation}
with $f(x,t)=\ex^{-\pi^2 (1+t)} \sin(\pi x) (1+t),$ $\alpha =0.5,$ $\varphi(x)=\left( 1+0.5 \ex^{-2 \pi^2}\right) \sin(\pi x),$ $q(x,t)=1+t.$ Then, the solution of this problem is $u(x,t)=\ex^{-\pi^2 (1+t)} \sin(\pi x).$

The problem \eqref{nle-pr} can be write down in the form \eqref{nle-post2} where the operator $A(t)$ is defined by 
\begin{equation}\label{nle-defA}
\begin{split}
& D(A(t))=D(A)=\{v \in H^2(0,1): v(0)=0, \, v(1)=0 \}, \\
&A(t)v=-\frac{\partial^2 v}{\partial x^2}+(1+t)v.
\end{split}
\end{equation}
Coefficients of the system \eqref{nle-nslar} were calculated by using the Fourier series expansion. The results of calculation are presented in tables confirm our theory above.

\begin{table*}[h]
    \begin{center}
        \begin{tabular}{|c|c|}
        \hline
      Point $t$ & $\varepsilon$ \\
      \hline
            -1 & 0.00005276 \\
            \hline
            -0.70710678 & 0.00097645 \\
            \hline
            0 & 0.00063440 \\
            \hline
            0.70710678 & 0.00029592 \\
            \hline
            1 & 0.00010552 \\
            \hline
        \end{tabular}
    \end{center}

    \caption{The error in the case $n=4,$ $x=0.5$}
    \label{tab:M4T1}
\end{table*}

\begin{table*}[h]
    \begin{center}
        \begin{tabular}{|c|c|}
        \hline
      Point $t$ & $\varepsilon$ \\
      \hline
            -1 & 8.12568908Ee-7 \\
            \hline
            -0.86602540 & 0.00010146 \\
            \hline
            -0.5 & 0.00030932 \\
            \hline
            0 & 0.00022136\\
            \hline
            0.5 & 0.00013419\\
            \hline
            0.86602540 & 0.00007182\\
            \hline
            1 & 0.00000162 \\
            \hline
        \end{tabular}
    \end{center}

    \caption{The error in the case $n=6,$ $x=0.5$}
    \label{tab:M6T1}
\end{table*}

\begin{table*}[h]
    \begin{center}
        \begin{tabular}{|c|c|}
        \hline
     Point $t$& $\varepsilon$ \\
      \hline
            -1 & 0.00000117 \\
            \hline
            -0.92387953 & 0.00000613 \\
            \hline
            -0.70710678 & 0.00004544 \\
            \hline
            -0.38268343 & 0.00005753 \\
            \hline
             0 & 0.00004745 \\
            \hline
             0.38268343 & 0.00003362 \\
            \hline
             0.70710678 & 0.00002096 \\
            \hline
             0.92387953 & 0.00000846 \\
            \hline
            1 & 0.00000235 \\
            \hline
        \end{tabular}
    \end{center}

    \caption{The error in the case $n=8,$ $x=0.5$}
    \label{tab:M8T1}
\end{table*}

\begin{table*}[h]
    \begin{center}
        \begin{tabular}{|c|c|}
        \hline
     Point $t$& $\varepsilon$ \\
      \hline
            -1 & 0.49451310e-8 \\
            \hline
            -0.96592582 & 0.14687232e-7 \\
            \hline
            -0.86602540 & 0.23393074e-6 \\
            \hline
            -0.70710678 & 0.54494052e-6 \\
            \hline
             -0.5 & 0.76722515e-6 \\
            \hline
            -0.25881904 & 0.82803283e-6 \\
            \hline
             0  & 0.76362937e-6 \\
            \hline
             0.25881904 & 0.63174173e-6 \\
            \hline
             0.5 & 0.47173110e-6 \\
            \hline
             0.70710678 & 0.30381367e-6\\
            \hline
             0.86602540 & 0.14341583e-6 \\
            \hline
             0.96592582 & 0.21271757e-7 \\
            \hline
             1 & 0.98902621e-8 \\
            \hline
        \end{tabular}
    \end{center}

    \caption{The error in the case $n=12,$ $x=0.5$}
    \label{tab:M12T1}
\end{table*}

\begin{table*}[h]
    \begin{center}
        \begin{tabular}{|c|c|}
        \hline
     Point $t$ & $\varepsilon$ \\
   \hline
            -1 &  0.20628738e-11 \\
            \hline
             -0.98078528 & 0.28602854e-10 \\
            \hline
             -0.92387953 & 0.48425552e-9 \\
            \hline
             -0.83146961 & 0.14258845e-8 \\
            \hline
             -0.70710678 & 0.25968220e-8 \\
            \hline
             -0.55557023 & 0.36339719e-8 \\
            \hline
             -0.38268343 & 0.42916820e-8 \\
            \hline
             -0.19509032 & 0.44975339e-8 \\
            \hline
             0 & 0.43045006e-8 \\
            \hline
             0.19509032 & 0.38169887e-8 \\
            \hline
             0.38268343 & 0.31414290e-8 \\
            \hline
             0.55557023 & 0.23686579e-8 \\
            \hline
             0.70710678 & 0.15787207e-8 \\
            \hline
             0.83146961 & 0.85640040e-9 \\
            \hline
             0.92387953 & 0.30309439e-9 \\
            \hline
             0.98078528 & 0.16809109e-10 \\
            \hline
             1 & 0.41257476e-11  \\
            \hline
        \end{tabular}
    \end{center}

    \caption{The error in the case $n=16,$ $X=0.5$}
    \label{tab:M16T1}
\end{table*}

{\bf Acknowledgment}. The authors would like to acknowledge the
support provided by the Deutsche Forschungsgemeinschaft (DFG).


\end{document}